\documentclass[a4paper, 12pt]{article}
\usepackage{tgheros}
\usepackage{amsmath, amsthm, latexsym, amssymb, amsfonts, epsf, xcolor, dsfont}
\newtheorem{conj}{Conjecture}

\newtheorem{theorem}{Theorem}[section]
\newtheorem{lemma}[theorem]{Lemma}

\newtheorem{corollary}[theorem]{Corollary}

\newtheorem*{theorem*}{Theorem}

\theoremstyle{definition}

\usepackage{mathtools} 
\DeclarePairedDelimiter\abs{\lvert}{\rvert}%
\DeclarePairedDelimiter\norm{\lVert}{\rVert}%

\makeatletter
\let\oldabs\abs
\def\abs{\@ifstar{\oldabs}{\oldabs*}}
\let\oldnorm\norm
\def\norm{\@ifstar{\oldnorm}{\oldnorm*}}
\makeatother

\DeclareMathOperator{\F}{\mathbb{F}}

\newcommand{\PP}{\mathbb{P}}

\newcommand{\Fq}{\F_q}

\begin{document}
\baselineskip=16.3pt
\parskip=14pt

\begin{center}
\section*{Proof of a Conjecture on Primitive Quartic Polynomials over Finite Fields}

{\large 
Rod Gow and Gary McGuire
 \\ { \ }\\
School of Mathematics and Statistics\\
University College Dublin\\
Ireland}
\end{center}

 \subsection*{Abstract}
 
We present a proof of a conjecture we made about primitive polynomials,
 for $q$ sufficiently large.
The argument follows a standard route, however a deeper result than usual with a character sum
bound is required. This is a theorem of Fu and Wan.
The proof is  AI assisted.
 
 MSC 11T06
 
 Keywords primitive polynomial
 
 \newpage

\section{Introduction}

Let $q$ be a prime power. Let $\alpha \in \mathbb{F}_{q^2} \setminus \mathbb{F}_q$ be fixed.

For each $\lambda \in \F_q$ let $y_\lambda $ be one of the two elements of $\mathbb{F}_{q^4}$
satisfying  $$y_\lambda^2+y_\lambda+\lambda=\alpha.$$
Then $y_\lambda $ will lie in $\mathbb{F}_{q^4}$, and possibly in $\mathbb{F}_{q^2}$.
Indeed, the discriminant of the quadratic is
$ 1 - 4( \lambda - \alpha) \in \mathbb{F}_{q^2}$, so $y_\lambda \in \mathbb{F}_{q^4} \setminus \mathbb{F}_{q^2}$  if and only if 
$ 1 - 4\lambda +4\alpha$ is a nonsquare in  $\mathbb{F}_{q^2}$.
It is not hard to show that $ 1 - 4\lambda +4\alpha$ is a nonsquare in  $\mathbb{F}_{q^2}$ for $(q+1)/2$ values of $\lambda$.
In a recent paper \cite{GM} we conjectured that  $y_\lambda $ is a primitive element of $\mathbb{F}_{q^4}$
for one of these $(q+1)/2$ values of $\lambda$. Here is an equivalent statement of Conjecture 1 in \cite{GM}.

\begin{conj}\label{oddconj2}
Let $q\not=13$ be any odd prime power.
Let $\alpha \in \mathbb{F}_{q^2} \setminus \mathbb{F}_q$ be fixed.
Then there exists $\lambda \in \F_q$ such that
\[
x^2+ x+\lambda - \alpha
\]
is a primitive quadratic polynomial in $\F_{q^2}[x]$.
\end{conj}

In this paper we present a proof of Conjecture \ref{oddconj2},
 for $q$ sufficiently large. To be precise, we prove that there exists a constant $q_0$ such that
 Conjecture \ref{oddconj2} is true for all odd prime powers $q>q_0$. We do not attempt to find $q_0$.
 We were assisted by AI, as described in the AI declaration at the end of the article.

A proof of Conjecture \ref{oddconj2} by Zhou and Wu has been posted on the arXiv recently \cite{ZW}.
The authors did not use AI.
The proof is for all $q$, not just for all $q$ sufficiently large.

In section 2 we outline the first steps of the proof, which is a standard method for proving the existence of  primitive elements in certain subsets
of $\F_{q^n}$.
The normal route is to bound the character sum over the subset, and then use estimates of $\omega(q^n-1)$ to finish. We prove the (crude) estimate we need in section 3. In section 4 we present the remainder of the proof, assuming the character sum bound. Finally, in section 5 we prove the character sum bound. This uses a deep result of Fu and Wan \cite{FW}.

With better estimates of $\omega(q^n-1)$ one could lower the implied constant $q_0$. We make no attempts to estimate
or improve $q_0$ in this article.

\section{First steps}

Let $y\in \F_{q^4}$.
Define the indicator function of primitive elements by
$$\mathds{1}(y)=\begin{cases}
        1 & \text{if $y$ is a primitive element,}\\
        0 & \text{otherwise}.
    \end{cases}$$

From a result of Vinogradov (see exercise 5.14 in \cite{LN}) we have

$$\mathds{1}(y)=\frac{\varphi(q^4-1)}{q^4-1} \sum_{d \mid (q^4 - 1)} \frac{\mu(d)}{\varphi(d)} \sum_{\chi : \chi^d=1 } \chi(y)
    $$

where $\mu$ is the Möbius function, $\varphi$ is Euler's totient function, and $\chi$ is a multiplicative character of $\mathbb{F}_{q^4}^*$.

For any subset $M$ of $\F_{q^4}$ it is clear that
$$\sum_{y\in M} \mathds{1}(y) = \text{number of primitive elements in $M$}.$$

We let $$M=\bigcup_{\lambda \in \F_q} \text{(roots of $x^2+x+\lambda - \alpha$)}\subseteq \F_{q^4}$$
(so $|M|=2q$) and by an abuse of notation we define
$$M(\chi)
=\sum_{y\in M} \chi(y).$$
Then
\[
N:= \text{number of primitive elements in $M$}=
\frac{\varphi(q^4-1)}{q^4-1} \sum_{d \mid (q^4 - 1)} \frac{\mu(d)}{\varphi(d)} \sum_{\chi : \chi^d=1 } M (\chi).
\]
We want to prove that  $N>0$.

For any subset $S\subseteq M$ we define 
$$S(\chi)
=\sum_{y\in S} \chi(y)$$
and define
\begin{equation}\label{ns}
N_S:=\frac{\varphi(q^4-1)}{q^4-1} \sum_{d \mid (q^4 - 1)} \frac{\mu(d)}{\varphi(d)} \sum_{\chi : \chi^d=1 } S (\chi).
\end{equation}
Then $S$ contains a primitive element if and only if $N_S>0$. Clearly if $N_S>0$ then $N>0$.
We will choose a particular subset $S$ to prove that $N_S>0$, and to thereby prove Conjecture \ref{oddconj2}, for all $q$ sufficiently large.

\section{Estimate of $\omega(n)$}

Let $\omega(n)$ be the number of distinct prime factors of $n$.
We will prove an inequality that we need later in the proof.

\begin{lemma}\label{omega}
Let $c>1$ be a real constant. Then
\[
n>c\left( 2^{\omega(n^8 - 1)}-1 \right)
\]
for all $n$ sufficiently large. 
\end{lemma}
In fact,  stronger statements are known to be true, so we can prove what we need with elementary estimates.

\begin{proof}
We use a well known theorem (\cite{T}
Chapter I.5, Section 5.3) which states that
\[ \omega(m) < 2 \frac{\log m}{\log \log m} \]
for all $m$ sufficiently large (in fact \cite{T} has $1+o(1)$ instead of 2).

Replacing $m$ with $n^8-1$, and replacing $\log (n^8 -1)$ with $\log (n^8)$, 
and replacing $\log \log (n^8 -1)$ with $\log \log n$, we get
\[
\omega(n^8-1) < 16 \frac{\log n}{\log \log n}.
\] 
Then 
\[
2^{\omega(n^8-1)} < 2^{ \frac{16\log n}{\log \log n}}=e^{ \frac{16\log n \log 2}{\log \log n}}
=n^{ \frac{16 \log 2}{\log \log n}}.
\]
The exponent $\frac{16 \log 2}{\log \log n}$ goes to 0 as $n\rightarrow \infty$,
so $\frac{16 \log 2}{\log \log n}<1/2$ for all $n$ sufficiently large.
(The choice of $1/2$ could be any number between 0 and 1.)
Therefore, for sufficiently large $n$, we may conclude
\[
2^{\omega(n^8-1)} <n^{1/2}.
\]
Multiplying by $c$ gives $c\cdot 2^{\omega(n^8-1)} <cn^{1/2}$ and the simple inequality
$cn^{1/2}<n+c$ (for all $n$ sufficiently large) completes the proof.
\end{proof}

\section{Proof of Conjecture Assuming Character Sum Bound}

Recall that $S$ is a nonempty subset of $M$, to be chosen later.
In this section we present a proof that $N_S>0$, assuming a bound on character sums. 
We will then prove this character sum bound in the next section.

In the expression \eqref{ns} for $N_S$ we separate out the term with $d=1$ (which is equal to $q$) from the other terms, to get

$$N_S = \frac{\varphi(q^4-1)}{q^4-1} \biggl( q + E \biggr)$$
where
\begin{equation}\label{e1}
E =  \sum_{d \mid (q^4-1), d > 1} \frac{\mu(d)}{\varphi(d)} \sum_{\chi : \chi^d=1 } S(\chi).
\end{equation}

To prove that $N_S>0$ it suffices to prove that $|E|<q$.

Taking absolute values in \eqref{e1} and using the triangle inequality gives
\begin{equation}\label{e2}
\left\vert{} E \right\vert{} \leq \sum_{d \mid (q^4-1), d > 1} \frac{\vert{}\mu(d)\vert{}}{\varphi(d)} \sum_{\chi : \chi^d=1} |S(\chi)|.
\end{equation}

In section \ref{bdpf}  we will prove  in Corollary \ref{finalc} that, for a certain subset
$S_1 \subseteq M$, there exists a constant c such that 
\begin{equation}\label{char}
\left\vert{} S_1(\chi) \right\vert{} \le c\sqrt{q}.
\end{equation}

Assuming this is true, applying this character sum bound to \eqref{e2} gives

$$\left\vert{} E \right\vert{} \leq \sum_{d \mid (q^4-1), d > 1} \frac{\vert{}\mu(d)\vert{}}{\varphi(d)} \sum_{\chi : \chi^d=1} c\sqrt{q}.$$
Since there are exactly $\varphi(d)$ characters of order $d$, we get
$$\left\vert{} E \right\vert{} \leq \sum_{\substack{d \mid (q^4-1), d > 1}} \vert{}\mu(d)\vert{} c\sqrt{q}.$$

The terms with $\mu(d)=0$ can obviously be ignored --- they are the non-squarefree divisors.
So we restrict ourselves to squarefree divisors. We then have
\begin{equation}\label{e3}
\left\vert{} E \right\vert{} \leq \sum_{\substack{d \mid (q^4-1) \\ \mu(d) \neq 0, d > 1}} c\sqrt{q}.
\end{equation}

Let $\omega(m)$ be the number of distinct prime factors of $m$. This is a standard notation for a well known function in number theory. The number of squarefree divisors of $m$ is $2^{\omega(m)}$.
Putting $m=q^4-1$,
the total number of square-free divisors of $q^4 - 1$ is $2^{\omega(q^4-1)}$. Subtracting the $d=1$ case, there are $2^{\omega(q^4-1)} - 1$ divisors $d>1$ of $q^4-1$ with $\mu (d) \neq 0$. 
This is the number of terms in the sum \eqref{e3}.
So we have
$$\left\vert{} E \right\vert{} \leq  \left(2^{\omega(q^4-1)} - 1\right)\ c\sqrt{q}.$$

For $N_S$ to be strictly positive, we need $q > \vert{} E \vert{}$, so we are done if 
\[
q > \left( 2^{\omega(q^4-1)} - 1 \right) c\sqrt{q}.
\]
Therefore we have proved that 
$$N_S > 0 \quad \text{whenever} \quad \sqrt{q} > c\left( 2^{\omega(q^4 - 1)}-1 \right).$$
The inequality $\sqrt{q} > c\left( 2^{\omega(q^4 - 1)}-1 \right)$ is true for all $q$ sufficiently large, by Lemma \ref{omega}.
Therefore $N_S>0$ for all $q$ sufficiently large.
Conjecture \ref{oddconj2}  is now proved, for all $q$ sufficiently large, under the assumption that \eqref{char} is true.

In the last section we prove that \eqref{char} is true.

\section{Proof of Character Sum Bound}\label{bdpf}

First recall that
$$M(\chi)=\sum_{\lambda \in \F_q}\quad \sum_{\text{$\gamma$ root of $x^2+x+\lambda - \alpha$}} \chi(\gamma)
=\sum_{y\in M} \chi(y)$$
where $$M=\bigcup_{\lambda \in \F_q} \text{(roots of $x^2+x+\lambda - \alpha$)}.$$
Recall that 
 $\alpha \in \mathbb{F}_{q^2} \setminus \mathbb{F}_q$ is fixed, and
the objective is to prove that some subset $S\subseteq M$ contains a primitive element of $\F_{q^4}$.

Completing the square we write
\[
x^2+x+\lambda - \alpha=(x+\frac{1}{2})^2+\lambda - \alpha-\frac{1}{4}=z^2+\lambda-c
\]
where $z=x+\frac{1}{2}$ and $c=\alpha+\frac{1}{4}$.
Let 
$$T=\bigcup_{\lambda \in \F_q} \text{(roots of $z^2+\lambda - c$)}.$$
Obviously $T=M+\frac{1}{2}$, and both sets have $2q$ elements, and 0 is not in $M$ or $T$.
Of course, $M$ containing a primitive element and $T$ containing a primitive element are not the same thing.
However some calculations are easier with $T$. 

We now use the fact that elements of $T$ are elements of $\F_{q^4}$ whose square is in $\F_{q^2}$.
Let $d$ be a fixed nonsquare in $\F_{q^2}$, so that $\F_{q^4}=\F_{q^2}(\sqrt{d})$. 
Any $z\in \F_{q^4}$ can be written uniquely as $z=u+v\sqrt{d}$, where $u,v\in \F_{q^2}$.
Then $z^2=u^2+dv^2+2uv\sqrt{d}$.

Let $z=u+v\sqrt{d}\in T$, so then $z^2 \in \F_{q^2}$, so $2uv=0$.
So there are two types of elements of $T$, either 
\begin{enumerate}
\item $v=0$ and $z^2=u^2$, or
\item $u=0$ and $z^2=dv^2$.
\end{enumerate}
(The first case $v=0$ will eventually be discarded, because in this case  $z\in \F_{q^2}$ so $z$ and $z+\frac{1}{2}$ cannot possibly be primitive elements in $\F_{q^4}$.)

Any element of $\F_{q^2}$ can be written uniquely as $a+b\alpha$ for some $a,b\in \F_q$.
Let $\ell$ be the $\F_q$-linear functional on $\F_{q^2}$ defined by $\ell (a+b\alpha)=b$.
Note that $\ell$ vanishes on $\F_q$, and $\ell (\alpha)=1$, and $\ell(c)=\ell(\alpha+\frac{1}{4})=
\ell (\alpha)+\ell (\frac{1}{4})= 1+0=1$.

Let $z\in T$, so $z^2=c-\lambda $ for some $\lambda\in \F_q$.
Then $\ell(z^2)=\ell(c)-\ell(\lambda)=1-0=1$.
Therefore we may re-write our two types of elements of $T$ as
\begin{enumerate}
\item $T\cap \F_{q^2}$, elements $z=\pm u\in \F_{q^2}$ with  $\ell(u^2)=1$.
\item $T\cap \sqrt{d} \F_{q^2}$, elements $z=\pm \sqrt{d} v\in \F_{q^4}$ with  $\ell(dv^2)=1$.
\end{enumerate}

Considering $\F_{q^2}$ as a 2-dimensional vector space over $\F_q$, the functions
$u\mapsto \ell(u^2)$ and $v\mapsto \ell(dv^2)$ are non-degenerate quadratic forms.
At this point we discard the set $T\cap \F_{q^2}$, as mentioned earlier, because it can never contain a primitive element. 

The quadratic form $v\mapsto \ell(dv^2)$ is anisotropic, and therefore takes every nonzero value $q+1$ times.
In particular, this holds for the value 1, so $|T\cap \sqrt{d} \F_{q^2}|=q+1$.

(To see that $v\mapsto \ell(dv^2)$ is anisotropic, note that $\ell(dv^2)=0$ implies $dv^2\in \F_q$, which implies
that $d$ is a square in $\F_{q^2}$, a contradiction.)

This shows that $T\cap \sqrt{d} \F_{q^2}$ is a (non-degenerate smooth)  conic over $\F_q$. 
There is a standard method for finding a parametrization of a conic by rational functions, which works over any field. 
We will use this parametrization.

We move back to $M$ now, so let 
\begin{equation}\label{s1}
S_1=(T\cap \sqrt{d} \F_{q^2})-\frac{1}{2} \subseteq M.
\end{equation}

Then $|S_1|=q+1$. We will prove that $S_1$ contains a primitive element. 
We prove the following lemma, which is the rational parametrization that we need.

Let $C=\{y \in \F_{q^2} : \sqrt{d} y - \frac{1}{2} \in S_1\}$.

\begin{lemma}\label{param}
Continue the above notation.
There are polynomials $P\in \F_{q^4} [t]$ with $\deg P\le2$ and $D\in\Fq[t]$ with $D$ an irreducible
quadratic, and an element $P_{\infty}\in S_{1}$, such that
\[
S_{1}=\Bigl\{\,R(t):=\frac{P(t)}{D(t)}\ :\ t\in\Fq\Bigr\}\ \cup\ \{P_{\infty}\},
\]
the $q$ values $R(t)$, $t\in\Fq$, being pairwise distinct.  

Moreover $R(t)$ is a nonconstant
rational function whose pole divisor on $\PP^{1}({\overline{\F_q}})$ is either the sum of two
distinct simple poles or a single simple pole.
\end{lemma}

\begin{proof}
Let $Q: v\mapsto \ell(dv^2)$ be the anisotropic quadratic form on $\F_{q^2}$, as above.
Let $B$ be its polarization, the symmetric $\Fq$-bilinear form
\[
B:\F_{q^2}\times \F_{q^2}\longrightarrow\Fq,\qquad B(y,z):=\ell(d yz),
\]
 so that
\begin{equation}\label{eq:polar}
Q(y)=B(y,y)\quad\text{and}\quad
Q(y+sz)=Q(y)+2s\,B(y,z)+s^{2}Q(z)
\end{equation}
for all $y,z\in \F_{q^2},\ s\in\Fq$.
  Two facts will be used repeatedly. The first is that
\begin{equation}\label{eq:aniso}
Q(z)\neq0\quad\text{for every }z\in \F_{q^2}\setminus\{0\}\qquad(\text{i.e. $Q$ is anisotropic}).
\end{equation}
The second fact is that $B$ is nondegenerate, i.e.,
\begin{equation}\label{eq:nondeg}
\text{for } y\neq0 \text{ the linear form } z\mapsto B(y,z) \text{ on } \F_{q^2} \text{ is nonzero}.
\end{equation}

Let $C=\{y \in \F_{q^2} : \sqrt{d} y - \frac{1}{2} \in S_1\}$.

Since $C$ is nonempty we may fix $y_{0}\in C$, so
$Q(y_{0})=1$.  Note $y_{0}\neq0$, since $Q(0)=0\neq1$.

Let
\[
\PP^{1}(\Fq):=\bigl(\F_{q^2}\setminus\{0\}\bigr)\big/\Fq^{\times},
\]
be the set of $\Fq$-lines through $0$ in the $2$-dimensional $\Fq$-space $\F_{q^2}$.
 We write $[z]$ for
the line through $z\in \F_{q^2}$, $z\neq0$.

For $z\in \F_{q^2}$, $z\neq0$, put
\[
s_{z}:=-\,\frac{2\,B(y_{0},z)}{Q(z)}\in\Fq ,
\]
noting $Q(z)\not=0$ by \eqref{eq:aniso}.  We claim that
\[
\pi:\PP^{1}(\Fq)\longrightarrow C,\qquad \pi([z]):=y_{0}+s_{z}z
\]
is a well-defined bijective map.  First, by \eqref{eq:polar},
\[
Q(y_{0}+s_{z}z)=1+2s_{z}B(y_{0},z)+s_{z}^{2}Q(z)
=1-\frac{4\,b(y_{0},z)^{2}}{Q(z)}+\frac{4\,b(y_{0},z)^{2}}{Q(z)}=1,
\]
so indeed $\pi([z])\in C$.  Secondly, the value does not depend on the representative. For
$c\in\Fq^{\times}$ we have $B(y_{0},cz)=c\,B(y_{0},z)$ and $Q(cz)=c^{2}Q(z)$, and so
$s_{cz}=c^{-1}s_{z}$ and $y_{0}+s_{cz}(cz)=y_{0}+s_{z}z$.

Geometrically, $s_{z}$ is the unique parameter $s$ for which the line $\{y_{0}+sz:s\in\Fq\}$ meets
$C$ a second time, because
\[
Q(y_{0}+sz)=1\iff s\bigl(2B(y_{0},z)+sQ(z)\bigr)=0
\iff s=0\ \text{ or }\ s=s_{z},
\]
again by \eqref{eq:aniso}.  (When $B(y_{0},z)=0$ the two solutions coincide, and the line is
tangent to $C$ at $y_{0}$.)

\medskip
\noindent\emph{To show $\pi$ is surjective.}
Let $y\in C$.  If $y\neq y_{0}$, set $z:=y-y_{0}\neq0$.  Then, by
\eqref{eq:polar} with $s=1$,
\[
1=Q(y)=Q(y_{0}+z)=1+2B(y_{0},z)+Q(z),
\]
so $2B(y_{0},z)=-Q(z)$ and therefore $s_{z}=-2B(y_{0},z)/Q(z)=1$. Consequently
$\pi([z])=y_{0}+z=y$.  If $y=y_{0}$,  the nonzero $\Fq$-linear form
$z\mapsto B(y_{0},z)$ on the $2$-dimensional space $\F_{q^2}$ has a $1$-dimensional kernel by nondegeneracy, so there
is $z_{\tau}\neq0$ with $B(y_{0},z_{\tau})=0$. Then $s_{z_{\tau}}=0$ and
$\pi([z_{\tau}])=y_{0}$.

\emph{To show $\pi$ is injective.}  Suppose $\pi([z])=\pi([w])=:y$ with $z,w\neq0$.  If $y\neq y_{0}$, then
$s_{z}\neq0\neq s_{w}$ and $y-y_{0}=s_{z}z=s_{w}w$, so $w=(s_{z}/s_{w})z$ with
$s_{z}/s_{w}\in\Fq^{\times}$, i.e.\ $[z]=[w]$.  If $y=y_{0}$, then $s_{z}z=0=s_{w}w$ forces
$s_{z}=s_{w}=0$, i.e.\ $B(y_{0},z)=B(y_{0},w)=0$. Thus $z$ and $w$ both lie in the
kernel of the form $z\mapsto B(y_{0},z)$, and so $[z]=[w]$ because the kernel is $1$-dimensional.

(We remark that this bijection reproves that  $|C|=q+1$.)

Next we choose  affine coordinates on $\PP^{1}(\Fq)$.
Fix an $\Fq$-basis $z_{1},z_{2}$ of $\F_{q^2}$ and set $z(t):=z_{1}+tz_{2}$ for $t\in\Fq$.  Every
class in $\PP^{1}(\Fq)$ is represented by exactly one of $z(t)$ $(t\in\Fq)$ or $z_{2}$.   Thus
$t\longmapsto[z(t)]$ is a bijection
\[
 \Fq\ \longrightarrow\ \PP^{1}(\Fq)\setminus\{[z_{2}]\}.
\]

Expanding by bilinearity, the following lie in $\Fq[t]$:
\[
D(t):=Q\bigl(z(t)\bigr)=Q(z_{1})+2t\,B(z_{1},z_{2})+t^{2}Q(z_{2}),
\]
\[
N(t):=B\bigl(y_{0},z(t)\bigr)=B(y_{0},z_{1})+t\,B(y_{0},z_{2}),
\]
with $\deg D=2$ (its leading coefficient is $Q(z_{2})\neq0$ by \eqref{eq:aniso}) and
$\deg N\le1$.  Define
\[
Y(t):=y_{0}D(t)-2N(t)\,z(t)\in \F_{q^2}[t].
\]
Note $\deg Y(t)\le2$.
For $t\in\Fq$ we have $D(t)=Q(z(t))\neq0$ by \eqref{eq:aniso} (note $z(t)\neq0$), and
\[
\pi\bigl([z(t)]\bigr)=y_{0}+s_{z(t)}z(t)=y_{0}-\frac{2N(t)}{D(t)}z(t)=\frac{Y(t)}{D(t)}.
\]
Since $D(t)$ has no root in $\Fq$, again by \eqref{eq:aniso},  $D(t)$ is an irreducible quadratic in
$\Fq[t]$.  

Finally, with this setup, we obtain a parametrization of $S_{1}$.
The map $\iota:\F_{q^2}\to \F_{q^4}$ defined by $\iota(y):=\sqrt{d}\,y-\tfrac12$, is injective  and by definition $S_{1}=\iota(C)$.  
Define
\[
P(t):=\sqrt{d}\,Y(t)-\tfrac12D(t)\in \F_{q^4}[t].
\]
Note $ \deg P(t)\le2$. Define
\[
R(t):=\frac{P(t)}{D(t)}\in \F_{q^4}(t).
\]
For $t\in\Fq$ we get $R(t)=\sqrt{d}\, \frac{Y(t)}{D(t)} -\tfrac12=\iota\bigl(\pi([z(t)])\bigr)$.  Since
$$\iota\circ\pi: \PP^{1}(\Fq)\longrightarrow C \longrightarrow S_{1}$$ 
is a bijection,  and
$t\mapsto[z(t)]$ is a bijection of $\Fq$ onto $\PP^{1}(\Fq)\setminus\{[z_{2}]\}$, the $q$
values $R(t)$, $t\in\Fq$, are pairwise distinct and
\[
S_{1}=\{R(t):t\in\Fq\}\ \cup\ \{P_{\infty}\}
\]
where
\[
P_{\infty}:=\iota\bigl(\pi([z_{2}])\bigr)=\sqrt{d}\Bigl(y_{0}+s_{z_{2}}z_{2}\Bigr)-\tfrac12 .
\]
$R(t)$ takes $q\ge3$ distinct values on $\Fq$, so it is not a constant rational function.

For the final statement, we must consider the pole divisor of $R(t)$.
Let $G:=\gcd(P,D)$ in $\F_{q^4}[t]$ (taken monic) and write $P=G\widetilde P$, $D=G\widetilde D$, so
that $R=\widetilde P/\widetilde D$ with $\gcd(\widetilde P,\widetilde D)=1$.  
The case
$\widetilde D=1$ is impossible, because it would imply that $D=G$ and  $R$ is constant, contradicting the previous paragraph. 
Since $D$ is separable of degree $2$ with roots $\rho\neq\rho^{q}$,
$\widetilde D$ is, up to a scalar, $D$, or $t-\rho$, or $t-\rho^{q}$.   
In each case $\widetilde D$ is separable, so the finite poles of $R$ are simple and are among
$\{\rho,\rho^{q}\}$.  Finally $R$ has no pole at $t=\infty$, because
\[
\operatorname{ord}_{\infty}(R)=\deg\widetilde D-\deg\widetilde P
=(\deg D-\deg G)-(\deg P-\deg G)=\deg D-\deg P\ge0 .
\]
Hence the pole divisor of $R(t) $ on $\PP^{1}({\overline{\F}_q})$ is $(\rho)+(\rho^{q})$, or $(\rho)$,
or $(\rho^{q})$, and is therefore a nonzero sum of distinct simple poles, as asserted.
\end{proof}

\begin{corollary}\label{s1ex}
There are polynomials $P\in \F_{q^4} [t]$ with $\deg P\le2$ and $D\in\Fq[t]$ with $D$ an irreducible
quadratic, and an element $P_{\infty}\in S_{1}$, such that either
$$S_1(\chi)=\sum_{a\in\F_q}\chi\biggl(\frac{P(a)}{D(a)} \biggr)+\chi(P_\infty)$$
or
$$S_1(\chi)=\sum_{\substack{a \in \mathbb{F}_q \\ P(a) \neq 0}} \chi\biggl(\frac{P(a)}{D(a)} \biggr)+1+\chi(P_\infty).$$
\end{corollary}

\begin{proof}
By definition, $$S_1(\chi)
=\sum_{y\in S_1} \chi(y).$$
If $P(t)$ has no roots in $\F_q$ then 
$$S_1(\chi)=\sum_{a\in\F_q}\chi\biggl(\frac{P(a)}{D(a)} \biggr)+\chi(P_\infty)$$
by Lemma \ref{param}. If $P(t)$ has one root in $\F_q$ then
$$S_1(\chi)=\sum_{\substack{a \in \mathbb{F}_q \\ P(a) \neq 0}} \chi\biggl(\frac{P(a)}{D(a)} \biggr)+1+\chi(P_\infty)$$
by Lemma \ref{param}. P(t) cannot have two roots in $\F_q$ because the values of $R(t)$ are distinct.
\end{proof}

The proof of \eqref{char} uses the following generalization of a theorem of Weil.
This is Theorem 5.5 in \cite{FW}.

\begin{theorem}[Fu-Wan \cite{FW}]\label{fw}
Let $f(t) \in \mathbb{F}_{q^d}(t)$ be a rational function. Write $f(t) = \prod_{j=1}^k f_j(t)^{n_j}$, where $f_j(t) \in \mathbb{F}_{q^d}[t]$ are irreducible polynomials and $n_j$ are non-zero integers. Let $\chi : \mathbb{F}_{q^d}^* \to \overline{\mathbb{Q}}_l^*$ be a multiplicative character for $\mathbb{F}_{q^d}$. Suppose that the rational function $\prod_{i=0}^{d-1} f(t^{q^i})$ is not of the form $h(t)^{\operatorname{ord}(\chi)}$ in $\overline{\mathbb{F}}(t)$, where $\operatorname{ord}(\chi)$ is the smallest integer $d$ such that $\chi^d = 1$. Then we have
\[
\left| \sum_{\substack{a \in \mathbb{F}_q \\ f(a) \neq 0, \infty}} \chi(f(a)) \right| \le \left( d \sum_{j=1}^k \deg(f_j) - 1 \right) \sqrt{q}.
\]
\end{theorem}

Finally then, we have the character sum bound that we need:

\begin{corollary}\label{finalc}
Let $S_1$ be as defined in \eqref{s1}. Then there exists a constant $c>0$ such that 
\[
|S_1(\chi)| \leq c \sqrt{q}.
\]
\end{corollary}

\begin{proof}
Taking absolute values in the statement of Lemma \ref{s1ex} and using the triangle inequality gives
\[
|S_1(\chi)| \leq \left| \sum_{\substack{a \in \mathbb{F}_q \\ P(a) \neq 0}} \chi\biggl(\frac{P(a)}{D(a)} \biggr) \right| +2.
\]
We may identity complex valued characters and $\overline{\mathbb{Q}}_l^*$ valued characters in this context, thanks to Deligne. 
Since $\frac{P(t)}{D(t)}$ has simple poles by  Lemma \ref{param}, $\prod_{i=0}^{3} \frac{P(t^{q^i})}{D(t^{q^i})}$ is not of the form $h(t)^d$ for any $d>1$.
Applying Theorem \ref{fw} gives
\[
|S_1(\chi)| \leq (4\cdot 4 -1)\sqrt{q} +2.
\]
Therefore $|S_1(\chi)| \leq 17\sqrt{q} $ so $c$ may be taken to be 17.
\end{proof}

\section{AI Declaration}

We were assisted by AI (Aristotle) especially with Lemma \ref{param}.   Aristotle proved its own character sum bound using etale cohomology. We  found a theorem in the literature (Theorem \ref{fw}) that can be used instead of the theorem that Aristotle proved.
No text in this article was written by AI.

\end{document}